\documentclass[12pt,reqno,a4paper]{amsart}
\usepackage{verbatim,rcs,cite, pstricks}
\usepackage{amssymb,amsfonts,latexsym,mathrsfs}
\usepackage{amsmath,amsthm,graphicx}
\usepackage[all]{xy}

\DeclareMathOperator{\Sym}{Sym}

\newtheorem{Theorem}{Theorem}

\newtheorem{theorem}{Theorem}[section]
\newtheorem{lemma}[theorem]{Lemma}
\newtheorem{proposition}[theorem]{Proposition}

\theoremstyle{definition}

\newcommand{\F}{\mathbb{F}}

\newcommand{\Z}{\mathbb{Z}}

\newcommand{\Cent}{\mathrm{Cent}}

\renewcommand{\theta}{\vartheta}
\renewcommand{\l}{\ell}

\DeclareMathOperator{\tr}{tr}

\title{On exceptional groups of order $p^5$}

\date{\today}

\author{John R. Britnell}
\address{Department of Mathematics, Imperial College London, South Kensington Campus, London SW7 2AZ, United Kingdom}
\email{j.britnell@imperial.ac.uk}

\author{Neil Saunders}
\address{Heilbronn Institute for Mathematical Research, University of Bristol, School of Mathematics, University Walk, Bristol BS8 1TW, United Kingdom}
\email{neil.saunders@bristol.ac.uk}

\author{Tony Skyner}
\address{Heilbronn Institute for Mathematical Research, University of Bristol, School of Mathematics, University Walk, Bristol BS8 1TW, United Kingdom}
\email{tony.skyner@bristol.ac.uk}


\begin{document}

\begin{abstract}
A finite group $G$ is \emph{exceptional} if it has a quotient $Q$ whose minimal faithful permutation degree is greater than that of $G$. We say that $Q$ is a \emph{distinguished} quotient. 

The smallest examples of exceptional $p$-groups have order $p^5$. For an odd prime $p$, we classify all pairs $(G,Q)$ where $G$ has order $p^5$ and $Q$ is a distinguished quotient. (The case $p=2$ has already been treated by Easdown and Praeger.) We establish the striking asymptotic result that as $p$ increases, the proportion of groups of order~$p^5$ with at least one exceptional quotient tends to $1/2$.
\end{abstract}

\maketitle

\section{Introduction}

Let $G$ be a finite group. The minimal degree $\mu(G)$ of $G$ is the least non-negative integer $n$ such that $G$ embeds into the symmetric group
$\Sym(n)$. The question of representing finite groups by permutations is one of the oldest in group theory, and the minimal degree $\mu(G)$ is a natural invariant of the group $G$. However it is well known that the function $\mu$ is badly behaved with regard to quotients: it is possible for a group $G$ to have a quotient $Q$ such that $\mu(G)<\mu(Q)$. Such cases are pitfalls in the design and implementation of permutation group algorithms (see 
\cite{Neumann}). 

The following terminology is derived from Easdown and Praeger \cite{EasdownPraeger}. If $G$ has a normal subgroup $N$ such that $\mu(G)<\mu(G/N)$, then we say that $N$ is a \emph{distinguished subgroup} of $G$, and that $G/N$ a \emph{distinguished quotient}. The group $G$ is \emph{exceptional} if it has at least one distinguished quotient $Q$; and we say also that $G$ is an \emph{exceptional extension} of $Q$. 

For every prime $p$, the smallest exceptional $p$-groups have order $p^5$. In the case $p=2$ this is known from \cite{EasdownPraeger}, in which it is shown there are two exceptional groups of order $32$; in each case the exceptional quotient is the quaternionic group $Q_{16}$. In the case of any odd prime
$p$, an example of an exceptional group of order $p^5$ is given by Lemieux \cite{Lemieux}. An argument for the non-existence of exceptional groups of order~$p^k$ for $k<5$ had previously appeared in his MSc thesis \cite{LemieuxThesis}. Because this argument is unpublished, and furthermore contains some errors, we give an independent proof of this result below.  

In this paper we give a complete classification of the exceptional groups of order~$p^5$, together with their distinguished subgroups, for any prime $p$.  Our main result is as follows. 
\begin{Theorem}
Let $G$ be a group of order $p^5$, and let $Q$ be a quotient of $G$ such that $\mu(Q)>\mu(G)$. 
Then $|Q|=p^4$, and the pair $(G,Q)$ is one of those listed in Table \ref{table:results}. 
\end{Theorem}
Table \ref{table:results} gives an explicit presentation for each exceptional $G$ and distinguished quotient $Q$. For the sake of completeness, we include the results of Easdown and Praeger in the case $p=2$ where there is only one such $Q$, namely the quaternion $2$-group of order $16$. 

The number of exceptional groups of order $p^5$ is surprisingly large. It is perhaps worth noting that Easdown and Praeger \cite{EasdownPraeger} expressed doubt as to whether such groups existed at all for $p>2$, although they found exceptional groups of order $p^6$ for all primes $p$.
\begin{Theorem}\label{T:counting}
The total number of exceptional groups of order $p^5$ is
\[
\left\{\begin{array}{ll} 2 & \textrm{if $p=2$,} \\ 10 & \textrm{if $p=3$,}\\ p+6 & \textrm{otherwise.} \end{array}\right.
\]
\end{Theorem}

The number of groups of order $p^5$ is at most $2p+71$. Theorem \ref{T:counting} that the number of exceptional groups of order $p^5$ is $p+6$ for $p>3$.  Hence the proportion of groups of order $p^5$ which are exceptional is asymptotically $1/2$. This suggests that the established term \emph{exceptional} for these groups (introduced in \cite{EasdownPraeger}) may be less appropriate than it has previously seemed. It would be interesting to understand the asymptotics for exceptional groups of order $p^6$ and larger, but we have nothing to add on this question at present.

\subsection{Notation for the exceptional groups}

The exceptional groups $G$ are listed in Table \ref{table:results}.  
For each possible distinguished quotient $Q$ we first give a presentation of $Q$ itself; we then give presentations of the exceptional extensions of $Q$. Our presentations use the following convention: the generators of $G$ are taken to be preimages of those of $Q$ (and denoted by the same letters), together with a generator $n$ of $N$. The generator $n$ is always central, and we shall write ``$n\ \textrm{central}$'' in our presentations as a convenient shorthand for the commutator relations. 

The names given to the extensions in the table are purely for ease of reference, and are not used in the text of the paper. The names given here to the various quotients, however, are introduced and used during the course of our argument. 

In the case that $p=2$ there are only two exceptional groups, $G_1$ and $G_2$ in the table; the classification in this case appears in \cite{EasdownPraeger} and will not be treated here. For each odd prime $p$, we take $\alpha=\alpha_p$ to be a fixed quadratic non-residue modulo $p$. We shall require the Legendre symbol,
\[
\left(\frac{a}{p}\right) = \left\{ \begin{array}{ll} $0$ & \textrm{if $p$ divides $a$,} \\ $1$ & \textrm{if $a$ is a quadratic residue modulo $p$,} \\
$-1$ & \textrm{if $a$ is a quadratic non-residue modulo $p$.} \end{array} \right.
\]
Where a parameter $\lambda$ appears in a presentation, it represents an integer in the range $[-(p-1)/2,\dots,(p-1)/2]$ satisfying the conditions stated in the Notes column. Distinct values of $\lambda$ yield presentations of non-isomorphic groups. 

The groups named $G_4$ and $G_6$ appear twice in the table with different presentations, as these groups possess two non-isomorphic distinguished quotients. Furthermore the exceptional extensions of the groups $Q_1(p)$ and $Q_\alpha(p)$ are isomorphic in pairs; details of these isomorphisms are given in Section \ref{s:isomorphisms}.

\begin{table}
\caption{Exceptional groups of order $p^5$ arranged by distinguished quotients.}\label{table:results}
\resizebox*{\linewidth}{!}{
\begin{tabular}{lll}
Group 	 		& Presentation 																	  				& 	Notes 	\\ 	
   \hline \\
$Q_{16}$ 			& $\langle x,y \mid x^8=1,\ y^2=x^4,\ [x,y]=x^{-2} \rangle$													& $	p=2$		\\
\quad $G_1$    	 	& $\langle x,y,n \mid x^8=n^2=1,\ y^2=x^4n,\ n\ \textrm{central},\ [x,y]=x^{-2}\rangle$ 				  			& 			\\
\quad $G_2$         	& $\langle x,y,n \mid x^8=n^2=1,\ y^2=x^4,\ n\ \textrm{central},\ [x,y]=x^{-2}n\rangle$  				  			& 			\\ 
\\

$Q_{81}$ 			& $\langle x,y,z \mid x^9=y^3=z^3=1,\ [x,y]=1,\ [x,z]=y,\  [y,z]=x^{-3} \rangle$ 									& 	$p=3$	\\
\quad $G_3$  		& $\langle x,y,z,n \mid x^9=y^3=z^3=n^3=1,\ n\ \textrm{central}, [x,y]=n,\ [x,z]=y,\ [y,z]=x^{-3}n \rangle$ 				&		 	\\
\quad $G_4$         	& $\langle x,y,z,n \mid x^9=y^3=z^3=n^3=1,\ n\ \textrm{central},\ [x,y]=1,\ [x,z]=y,\ [y,z]=x^{-3}n\rangle$ 				&		 	\\ 
\\

$Q_1(3)$ 			&  $\langle x,y,z \mid x^{9}=y^3=1,\ z^{3}=x^{3},\ [x,y]=1,\ [x,z]=y,\ [y,z]=x^3 \rangle$								& 	$p=3$ 	\\
\quad $G_4$         	& $\langle x,y,z,n \mid x^9=y^3=n^3=1,\ n\ \textrm{central},\ z^3=x^3n,\ [x,y]=n,\ [x,z]=y,\ [y,z]=x^{3}n\rangle$  			& \\
\quad $G_5$         	& $\langle x,y,z,n \mid x^9=y^3=n^3=1,\ n\ \textrm{central},\ z^3=x^3n,\ [x,y],\ [x,z]=y,\ [y,z]=x^{3}n^2\rangle$ 			& \\ 
\quad $G_6$         	& $\langle x,y,z,n \mid x^9=y^3=n^3=1,\ n\ \textrm{central},\ z^3=x^3,\ [x,y]=1,\ [x,z]=y,\ [y,z]=x^{3}n\rangle$ 			& \\ 
\\

$Q_{\alpha}(3)$ 		&  $\langle x,y,z \mid x^{9}=y^3=1,\ z^3=x^{6},\ [x,y]=1,\ [x,z]=y,\ [y,z]=x^{6}\rangle$							& $p=3$, ($\alpha=2$) 		\\
\quad $G_6$         	&  $\langle x,y,z,n \mid x^9=y^3=n^3=1,\ n\ \textrm{central},\ z^3=x^{6}n,\ [x,y]=1,\ [x,z]=y,\ [y,z]=x^{6}\rangle$			&	\\ 
\quad $G_7$ 	 	&$\langle x,y,z,n \mid x^9=y^3=n^3=1,\ n\ \textrm{central},\ z^3=x^{6},\ [x,y]=1,\ [x,z]=y,\ [y,z]=x^{6}n\rangle$ 			& 			\\ 
\\

$Q(p)$ 			& $\langle x,y,z \mid x^{p^2}=y^p=z^p=[x,y]=[x,z]=1,\ [y,z]=x^p \rangle$										& 	$p$ odd	\\
\quad $E_1$       	& $\langle x,y,z,n \mid x^{p^2}=y^p=z^p=1,\ n\ \textrm{central},\ [x,y]=[x,z]=1,\ [y,z]=x^pn\rangle$ 					& 			\\
\quad $E_2$       	& $\langle x,y,z,n \mid x^{p^2}=z^p=1,\ y^p=n,\ n\ \textrm{central},\ [x,y]=[x,z]=1,\ [y,z]=x^pn\rangle$ 					& 			\\
\quad $E_3$       	& $\langle x,y,z,n \mid x^{p^2}=y^p=1,\ z^p,\ n\ \textrm{central},\ [x,y]=1,\ [x,z]=n,\ [y,z]=x^pn\rangle$ 					& 			\\
\quad $E_4$       	& $\langle x,y,z,n \mid x^{p^2}=y^p=1,\ z^p=n,\ n\ \textrm{central},\ [x,y]=1,\ [x,z]=n,\ [y,z]=x^pn\rangle$ 				& 			\\
\quad $E_5$       	& $\langle x,y,z,n \mid x^{p^2}=z^p=1,\ y^p=n,\ n\ \textrm{central},\ [x,y]=1,\ [x,z]=n,\ [y,z]=x^p\rangle$ 				& 			\\
\quad $E_6(\lambda) $       & $\langle x,y,z,n \mid x^{p^2}=z^p=1,\ y^p=n,\ n\ \textrm{central},\ [x,y]=1,\ [x,z]=n^\lambda,\ [y,z]=x^pn\rangle$
       &  $\lambda\neq 0$, $\left(\frac{1+4\lambda}{p}\right)=1$ \\ 
       \\

$Q_1(p)$ 			&  $\langle x,y,z \mid x^{p^2}=y^p=1,\ z^{p}=x^{p},\ [x,y]=1,\ [x,z]=y,\ [y,z]=x^p \rangle$							& 	$p>3$  	\\
\quad $F^{(1)}_1$     & $\langle x,y,z,n \mid x^{p^2}=y^p=1,\ z^p=x^p,\ n\ \textrm{central},\ [x,y]=1,\ [x,z]=y,\ [y,z]=x^pn\rangle$ 				& 			\\
\quad $F^{(1)}_2$     & $\langle x,y,z,n \mid x^{p^2}=y^p=1,\ z^p=x^p,\ n\ \textrm{central},\ [x,y]=n^{-1},\ [x,z]=y,\ [y,z]=x^pn^2\rangle$ 		& 			\\
\quad $F^{(1)}_3$      & $\langle x,y,z,n \mid x^{p^2}=y^p=1,\ z^p=x^pn,\ n\ \textrm{central},\ [x,y]=1,\ [x,z]=y,\ [y,z]=x^pn\rangle$ 			& 			\\
\quad $F^{(1)}_4(\lambda)$       & $\langle x,y,z,n \mid x^{p^2}=y^p=1,\ z^p=x^pn,\ n\ \textrm{central},\ [x,y]=n^{-1},\ [x,z]=y,\ [y,z]=x^pn^{1+\lambda}\rangle$ &
$\lambda\ge 0$, $\left(\frac{\lambda^2+4}{p}\right)=1$ \\
\quad $F^{(1)}_5(\lambda)$       & $\langle x,y,z,n \mid x^{p^2}=y^p=1,\ z^p=x^pn,\ n\ \textrm{central},\ [x,y]=n^{-\alpha},\ [x,z]=y,\ [y,z]=x^pn^{\alpha+\lambda}\rangle$ & $\lambda\ge 0$, $\left(\frac{\lambda^2+4\alpha}{p}\right)=1$ \\
\\

$Q_\alpha(p)$ &  $\langle x,y,z \mid x^{p^2}=y^p=1,\ z^p=x^{\alpha p},\ [x,y]=1,\ [x,z]=y,\ [y,z]=x^{\alpha p}\rangle$						& 	$p>3$	\\
\quad $F^{(\alpha)}_1$       & $\langle x,y,z,n \mid x^{p^2}=y^p=1,\ z^p=x^{\alpha p},\ n\ \textrm{central},\ [x,y]=1,\ [x,z]=y,\ [y,z]=x^{\alpha p}n\rangle$ 			& \\
\quad $F^{(\alpha)}_2$       & $\langle x,y,z,n \mid x^{p^2}=y^p=1,\ z^p=x^{\alpha p},\ n\ \textrm{central},\ [x,y]=n^{-1},\ [x,z]=y,\ [y,z]=x^{\alpha p}n^{(\alpha+1)}\rangle$ & \\
\quad $F^{(\alpha)}_3$       & $\langle x,y,z,n \mid x^{p^2}=y^p=1,\ z^p=x^{\alpha p}n,\ n\ \textrm{central},\ [x,y]=1,\ [x,z]=y,\ [y,z]=x^{\alpha p}n\rangle$ & \\
\quad $F^{(\alpha)}_4(\lambda)$       & $\langle x,y,z,n \mid x^{p^2}=y^p=1,\ z^p=x^{\alpha p}n,\ n\ \textrm{central},\ [x,y]=n^{-1},\ [x,z]=y,\ [y,z]=x^{\alpha p}n^{\alpha+\lambda}\rangle$ & $\lambda\ge 0$, $\left(\frac{\lambda^2+4\alpha}{p}\right)=1$ \\
\quad $F^{(\alpha)}_5(\lambda)$       & $\langle x,y,z,n \mid x^{p^2}=y^p=1,\ z^p=x^{\alpha p}n,\ n\ \textrm{central},\ [x,y]=n^{-\alpha},\ [x,z]=y,\ [y,z]=x^{\alpha p}n^{\alpha(\alpha+\lambda)}\rangle$ & $\lambda\ge 0$, $\left(\frac{\lambda^2+4}{p}\right)=1$\\ 
\end{tabular}
}
\end{table}

\subsection{Background}

Neumann \cite{Neumann} has pointed out that there exists a sequence of groups $(G_i)$ with quotients $(Q_i)$ such that $\mu(Q_i)$ grows exponentially with $\mu(G_i)$. Neumann's example takes $G_i$ to be the direct product, and $Q_i$ to be the central product, of $i$ copies of the dihedral group $D_8$. Holt and Walton \cite{HoltWalton} have proved the existence of a constant $c$ such that $\mu(G/N) \leq c^{\mu(G)-1}$ for all groups $G$ and normal subgroups $N$. The constant $c$ is shown to be less than $4.5$.  

The second author \cite{Saunders2} has recently described a general construction of exceptional groups which, like Neumann's construction, uses the idea of the central product as a quotient of the direct product. Instances of this construction are found to occur naturally in the context of binary polyhedral groups. 

Recall that the \emph{core} of a subgroup $H$ of $G$ is the intersection of the conjugates of $H$ in $G$. For a permutation representation of a group $G$ on a set $\Omega$, let $\{H_1,\dots,H_\l\}$ be a set of point stabilizers, one for each orbit of $G$ on $\Omega$. These subgroups determine the representation up to equivalence. The representation is faithful if and only
if the cores of the subgroups $H_1,\dots,H_\l$ intersect trivially. The degree of the representation is given by the sum $\sum_{i=1}^\l |G:H_i|$, and so the
minimal degree of $G$ can be expressed as the minimum value taken by this sum over collections of subgroups $\{H_i\}$ whose cores have trivial intersection (see \cite{Saunders} for a fuller treatment). 

We shall require the following theorem of Johnson \cite{Johnson}.
\begin{theorem}[Theorem 3 of \cite{Johnson}]\label{thm:johnson}
Let $G$ be a $p$-group whose centre $Z(G)$ is minimally generated by $d$ elements. Let $\{H_1, \ldots, H_\l\}$ be point stabilizers for a minimal representation of $G$. Then
\begin{enumerate}
\item if $p$ is odd then $\l=d$,
\item if $p=2$ then $\frac{d}{2}\leq \l \leq d$, with the bound $\l=d$ being achieved by some minimal representation of $G$.
\end{enumerate}
\end{theorem}

It will be helpful to collect together here a number of conditions which ensure that a quotient $Q$ of a group $G$ is not distinguished.
\begin{lemma}\label{lem:criteria}
Let $G$ be group, and $Q$ a quotient of $G$. If any one of the following conditions is met, then $Q$ is not a distinguished quotient of $G$.
\begin{enumerate}
\item $Q$ is isomorphic to a subgroup of $G$.
\item $G$ is abelian.
\item $Q$ is cyclic.
\item $Q$ is elementary abelian.
\item $G$ is a $p$-group where $p$ is an odd prime, and $Q$ has a maximal cyclic subgroup.
\end{enumerate}
\end{lemma}
\begin{proof}
\begin{enumerate}
\item A faithful representation of $G$ restricts to a faithful representation of any subgroup.
\item This follows from the fact that every quotient of a finite abelian group $G$ is isomorphic to a subgroup of $G$.
\item If $Q$ is cyclic then $G$ has a subgroup isomorphic to $Q$.
\item This is stated in \cite{KovacsPraeger} as a consequence of the main theorem.
\item This appears in \cite{LemieuxThesis}. The proof is short and so we give it here for convenience.
Suppose that $|Q|=p^n$ and that $Q$ has a cyclic subgroup $C$ of order $p^{n-1}$. It follows that $G$ itself has a cyclic subgroup subgroup $\overline{C}$ of order $p^{n-1}$, and in particular that $\mu(G)\ge p^{n-1}$.
We may assume that $Q$ is not itself cyclic, and hence that $Q$ is an extension of $C$ by a group $H\cong C_p$. It is well-known (see \cite[Theorem 4.1]{Brown} for instance) that any such extension is split (since $p$ is odd). If $Q$ is not abelian, then its non-normal factor $H$ is a core-free subgroup of index $p^{n-1}$, and so $\mu(Q)\le p^{n-1} \le \mu(G)$. The remaining case is that $Q\cong C\times H$, in which case $\mu(Q)=p^{n-1}+p$. Now if $Q$ is distinguished, then we must have $\mu(G)=p^{n-1}$. Hence a minimal degree representation of $G$ is transitive, and now by Theorem \ref{thm:johnson} it follows that $G$ has a cyclic centre, and hence that $G$ has a unique normal subgroup $Z$ of order $p$. But $Z$ is contained in the kernel of the canonical map onto $Q$, and hence intersects trivially with $\overline{C}$. So the subgroup $\overline{C}Z$ of $G$ is isomorphic to $C_{p^{n-1}}\times C_p\cong Q$.
\end{enumerate}\end{proof}

Kov\'acs and Praeger \cite{KovacsPraeger} have conjectured that no abelian quotient is distinguished. In later work \cite{KovacsPraeger2} they have shown that if a minimal counterexample $G$ exists to this conjecture, then $G$ is a $p$-group, with $Q$ being the quotient by the commutator subgroup $G'$, and with $\mu(Q)=\mu(G)+p$. Perhaps for this reason, the principal focus of work in this area has been on $p$-groups; however the conjecture remains open. Franchi \cite{Franchi} has shown that $\mu(G/G') \leq \mu(G)$ whenever $G$ contains a maximal abelian subgroup.  

The hypothesis that $p$ is odd in Lemma \ref{lem:criteria}(5) is necessary only to exclude the case that the quotient $Q$ is a generalized quaternion group; it is known from \cite{EasdownPraeger} that these groups arise as distinguished quotients. Excepting these groups, the result holds when $p=2$ as well.

\subsection{Organisation of the Paper}

In Section \ref{s:prelim} we set out some preliminary results. We show that there are no exceptional $p$-groups of order at most $p^4$, and that if~$G$ is an exceptional group of order $p^5$ with distinguished quotient $N$, then $N$ has order $p$. We classify the possible isomorphism classes of the quotient $Q=G/N$. Some of these results have appeared previously in \cite{LemieuxThesis}.
We also deal with exceptional groups of order $3^5$, which require special treatment. However groups of this order are easily handled by computer calculations, and we have not thought it necessary to give further justification for our results.  

In Section \ref{s:Q} we consider exceptional extensions of the group we have named $Q(p)$, in the case $p>3$. Similarly, in Section \ref{s:Q_1} we consider extensions of the two groups we have named $Q_1(p)$ and $Q_\alpha(p)$. These admit a common treatment, and indeed their exceptional extensions turn out to be the same up to isomorphism. 

\section{Preliminaries}\label{s:prelim}

\subsection{Groups of order at most $p^4$}
\begin{proposition}\label{prop:smallquotients}
Let $p$ be a prime, and let $Q$ be a distinguished quotient of a $p$-group~$G$. Then $|Q|>p^3$.
\end{proposition}
\begin{proof}
Let $G$ be a $p$-group and $Q$ a proper quotient of order at most $p^3$. Suppose first that $Q$ is abelian. Then since any abelian $p$-group of order $p^3$ is elementary abelian, or else has a maximal cyclic subgroup, we see from Lemma \ref{lem:criteria} that $Q$ is not distinguished.
But if $Q$ is  non-abelian then it is extraspecial of order $p^3$. If $Q$ is not the quaternion group $Q_8$, then $\mu(G/N)=p^2$; but $p^2$ is the smallest possible degree of a non-abelian $p$-group acting faithfully, and so $\mu(G)\ge\mu(G/N)$. If $Q=Q_8$ then $\mu(Q)=8$, and it is easy to check that no
$2$-group of degree less than $8$ has $Q$ as a quotient.
\end{proof}

It follows immediately from Proposition \ref{prop:smallquotients} that no $p$-group of order at most $p^4$ is exceptional.

\subsection{Distinguished quotients}\label{s: quotients}

Throughout this section, we suppose that $p$ is an odd prime. The following result is essentially due to Lemieux \cite{LemieuxThesis}. However he mistakenly excludes from his list the group that we have called $Q_{81}$ (called by him $G_{28}$). Also, one extra group in his list (which he calls $G_{29}$) is excluded by our arguments.
\begin{proposition} \label{proposition:quotients}
Suppose that $G$ is an $p$-group of order $p^5$ with a distinguished subgroup $N$. Then $N$ is cyclic and $G/N$ is isomorphic to one of the following groups.
\begin{eqnarray*}
Q_{81} &=& \langle x,y,z \mid x^9=y^3=z^3=1,\ [x,y]=1,\ [x,z]=y,\  [y,z]=x^{-3} \rangle\ \textrm{($p=3$)},\\
Q(p) &=& \langle x,y,z \mid x^{p^2}=y^p=z^p=[x,y]=[x,z]=1,\ [y,z]=x^p \rangle, \\
Q_1(p) &=& \langle x,y,z \mid x^{p^2}=y^p=1,\ z^{p}=x^{p},\ [x,y]=1,\ [x,z]=y,\ [y,z]=x^p \rangle, \\
Q_\alpha(p) &=& \langle x,y,z \mid x^{p^2}=y^p=1,\ z^p=x^{\alpha p},\ [x,y]=1,\ [x,z]=y,\ [y,z]=x^{\alpha p}\rangle,
\end{eqnarray*}
where $\alpha$ is a quadratic non-residue modulo $p$.
\end{proposition}
\begin{proof}
From Proposition \ref{prop:smallquotients} we see that $G/N$ has order $p^4$, and hence $|N|=p$. The groups of order $p^4$ are well known; a convenient reference for them is \cite{Lemieux}, who tabulates them with their minimal degrees. For the sake of brevity, we shall not here refer explicitly to individual groups.  

Let $Q=G/N$. Then Lemma \ref{lem:criteria} tells us that $Q$ is not elementary abelian, and contains no maximal cyclic subgroup. Since $G$ is non-abelian, we see that $\mu(G)\ge p^2$, and so we must have $\mu(Q)>p^2$. These considerations suffice to rule out all but six of the groups of order $p^4$ for odd $p$, aside from the four groups listed in the proposition. 

These six possibilities for $Q$ each have the property that $\mu(Q)\le 2p^2$. We must therefore have $\mu(G)<2p^2$, and it follows that $\mu(G)=p^2+cp$ for some $c<p$. Hence $G\cong H\times B$, where $H$ is a non-abelian permutation group acting on $p^2$ points and $B$ is elementary abelian, possibly trivial. The action of $H$ is a minimal degree representation, and it is transitive. So Theorem \ref{thm:johnson} tells us that $H$ has a unique central subgroup $N$, which clearly has order $p$. But now it is clear that $N$ must be the distinguished subgroup of $G$. This implies that $H$ is exceptional, and so $H=G$. 

Now we see that $G$ is a subgroup of $C_p \wr C_p$, and so $G$ has an elementary abelian maximal subgroup $V$. If $g\in G\setminus V$ then $g$ acts indecomposably on $V$, and its Jordan form comprises a single unipotent block of size $4$. It then follows that the quotient $G/N$ is isomorphic to an extension of a space of dimension $3$ by a linear map of order~$3$ acting indecomposably. In particular, the centre $Z(G/N)$ is cyclic, and so a minimal permutation representation of $G/N$ is transitive by Theorem \ref{thm:johnson}. But this is not the case for any of the six possible groups $Q$ under consideration, and so $G/N\not\cong Q$, which is a contradiction.
\end{proof}

\subsection{Commutators and $p$-th powers}

Recall the well known commutator identity
\begin{equation}\label{commid}
[x,yz] = [x,z][x,y]^z.
\end{equation}
In particular, if $G$ is nilpotent of class $2$, then commutators are central, and so $[x,yz]=[x,y][x,z]$.

The following fact is of great importance in what follows.

\begin{proposition}\label{prop:ppower}
Let $p>3$, and let $G$ be a $p$-group of nilpotency class at most $3$, such that $G'$ has exponent at most $p^2$. Then the map $g\mapsto g^p$ is an endomorphism of~$G$.
\end{proposition}
\begin{proof} Since commutators of weight $2$ are central in $G$, a straightforward induction establishes the following identity:
\[
(gh)^a = g^ah^a[h,g]^{a \choose 2}[[h,g],g]^{a \choose 3}[[h,g],h]^{{a \choose 3} + {a \choose 2}}.
\]
Now since $G'$ has exponent dividing $p^2$, it follows for $p>3$ that $(gh)^p=g^ph^p$.
\end{proof}

\subsection{Use of computation and the case $p=3$}\label{s:computation}
The classification of exceptional groups for $p=3$ stands alone since it does not comply with the general classification for $p\geq 5$.  Since groups of order $3^5$ are so readily dealt with computationally, we do not give full proofs of the results here. The classification of exceptional extensions of the groups $Q(p)$, $Q_1(p)$ and $Q_\alpha(p)$ for $p=3$ are presented in Table \ref{table:results}. We remark here that the groups $Q_{1}(3)$ and $Q_{81}$ possess a common exceptional extension, the group which we have called $G_4$. Also, $Q_{1}(3)$ and $Q_{\alpha}(3)$ possess a common exceptional extension, which we have called $G_6$. These groups appear with distinct presentations in Table \ref{table:results}, reflecting the particular quotient maps, but we have used a consistent nomenclature.

For the rest of the paper, we assume $p>3$. Our arguments depend no further on computation; however we have verified our classification of exceptional groups computationally for all primes $p$ up to $19$. 
For calculating minimal degrees we have used the procedure described in \cite{EliasSilbermanTakloo}.   
Our computations have been performed using the GAP \cite{GAP} and Magma \cite{Magma} computer algebra programs, and have made use of the library of small
$p$-groups in the Small Groups Database of Besche, Eick and O'Brien. 

\section{Extensions of $Q(p)$}\label{s:Q}

\subsection{First reduction}

Let $p$ be an odd prime, and let $Q$ be $Q(p)$ of order $p^4$, given by
\[
Q=\langle x,y,z \mid x^{p^2}=y^p=z^p= [x,y] = [x,z] = 1, [y,z]=x^p\rangle.
\]
A central extension of $Q$ by a group of order $p$ has the form

\begin{align}\label{eq:G16extension}
G&=\langle x,y,z,n \mid n^p=1, x^{p^2}=n^h, y^p=n^i, z^p=n^j,\nonumber \\ 
&\quad \quad \textrm{$n$ central}, [y,z] = x^pn^k, [x,z]=n^\l, [x,y]=n^m\rangle
\end{align}
for integers $i,j,k,\l,m\in\{0,\dots, p-1\}$. The generators of $G$ in this presentation have been so labelled that the images of $x,y,z$, in the quotient of $G$ by its central subgroup~$\langle n\rangle$, correspond to the generators $x,y,z$ of $Q$. We note that it is immediately clear from the presentation that $G$ has nilpotency class $2$, and that its derived group has exponent $p$. So Proposition \ref{prop:ppower} tells us that the $p$-power map is an endomorphism of~$G$.

It is clear that a subgroup $H$ of $Q$ not containing the socle $\langle x^p\rangle$ has order at most~$p$, from which it is clear that $\mu(Q)=p^3$. Suppose that $\mu(G)<\mu(Q)$; then it is clear that $G$ can have no element of order $p^3$. In particular,
the generator $x$ has order $p^2$, and so we must have $h=0$. In fact we may also assume that $m=0$, and our next step is to justify this assertion.
\begin{proposition}\label{prop:parameterelimination}
Let $G(i,j,k,\l,m)$ denote the group with the presentation (\ref{eq:G16extension}) above, with $h=0$. Then
for all $i,j,k,\l,m$ there exist $i',j',\l'$ such that $G(i,j,k,\l,m)\cong G(i',j',k,\l',0)$.
\end{proposition}
\begin{proof}
Let $X=\langle y,z,n\rangle$. Then $X$ has order $p^4$, and $Z(X)=\langle n,x^p\rangle$ has order $p^2$. Clearly the quotient $V=X/Z(X)$ is elementary abelian of order $p^2$. The derived subgroup $X'$ is equal to the cyclic group $\langle x^pn^k \rangle$.

There is a well defined map $\phi:V\longrightarrow \langle n\rangle$ given by
\[
\phi(v+Z(X)) = [x,v],\quad v \in X.
\]
There is also a map $\theta: V\times V\longrightarrow X'$ defined by
\[
\theta(u+Z(X),v+Z(X)) = [u,v].
\]
We may consider $V$ as a vector space, and the cyclic groups $\langle n\rangle$ and $X'$ as copies of the field $\F_p$. It is now apparent, by (\ref{commid}), that $\phi$ is a linear functional, and $\theta$ a non-degenerate alternating form, on the space $V$. If we pick $u$ to be a non-zero
element of the kernel of $\phi$, then there exists $v\in V$ such that $\theta(u,v)=x^pn^k$. Now we may construct a new presentation for the group $G$,
on generators $x,y',z',n$, where $y'+Z(X) = u$ and $z'+Z(X)=v$. It is straightforward to check that this presentation gives the isomorphism claimed, for some values of $i',j',\l'$.
\end{proof}

In the light of Proposition \ref{prop:parameterelimination}, we may drop the parameter $m$ from the notation introduced there, and write
$G(i,j,k,\l)$ for the group with the presentation (\ref{eq:G16extension}) with $h=m=0$. We call such a presentation a \emph{distinguished presentation} with parameters $(i,j,k,\l)$. The group affording such a presentation is a \emph{candidate extension} of $Q$; these include all of the exceptional extensions of $Q$.

\subsection{Structure of a candidate extension of $Q$}

We record here some observations on the structure of a candidate extension which will be needed later. Throughout this section $G$ is a candidate extension affording a distinguished presentation with generators $x,y,z,n$ and parameters $(i,j,k,\l)$.

\begin{proposition} $G$ has order $p^5$ (that is to say, it is a proper central extension of~$Q$).
\end{proposition}
\begin{proof} It is clear that $|G|$ is either $p^4$ or $p^5$. Consider the group $K$ with presentation
\begin{align*}
K &= \langle x,y,z,n \mid x^p=n^p=[x,y]=[x,n]=[y,n]=[z,n]=1,\\
& \quad \quad y^p=n^i, z^{p^2}=1, [y,z]=n^k, [x,z]=n^\l\rangle.
\end{align*}
The subgroup $L$ of $K$ generated by its elements $x,y,n$ is abelian of order $p^3$. It is easy to see that there is an automorphism of $L$ of order $p$
such that $x\mapsto xn^\l$, $y\mapsto yn^k$, $n\mapsto n$; this holds regardless of whether $y$ has order $p$ or $p^2$.
So $K$ is a semidirect product of $L$ by the group $\langle z\rangle$ of order $p^2$. The elements $n$ and $z^p$ are central in $K$. But we observe that the quotient $\overline{K}$ of $K$ by $\langle z^pn^{-j} \rangle$ has a presentation obtained from the presentation for $K$ simply by adding the relation $z^pn^{-j}=1$, and it can be checked that this is also a presentation for the quotient of $G$ by its subgroup $\langle x^p\rangle$. Now $x^p$ has order $p$ in $G$, since its image in $Q$ has order $p$, and it follows easily that $|G|=p|\overline{K}|=|K|=p^2|L|=p^5$ as required.
\end{proof}

\begin{proposition}\label{prop:centre}
If $\l\ne 0$ the the centre $Z(G)$ is elementary abelian of size $p^2$, and is equal to $\langle x^p,n\rangle$. If $\l=0$ then $Z(G)=\langle x,n\rangle$, which is isomorphic to $C_p\times C_{p^2}$.
\end{proposition}
\begin{proof}
It is clear that $A=\langle x^p,n\rangle$ is a central elementary abelian subgroup of order $p^2$. A set of representatives for the cosets of $A$ in $G$ is
$\{x^ay^bz^c \mid a,b,c\in \{0,\dots,p-1\}\}$. Now the element $x^ay^bz^c$ cannot commute with both $y$ and $z$ unless $b=c=0$, and so $Z(G)$ is contained in $\langle x,n\rangle$. The result follows from the facts that $x$ has order $p^2$, and that it is central if and only if $\l=0$.
\end{proof}

\begin{proposition}\label{prop:abeliansubgroup}
The subgroup $H=\langle x,y,n\rangle$ of $G$ is an abelian subgroup of order~$p^4$. If $\l\ne 0$ then $H$ is the unique abelian subgroup of $G$ of order $p^4$. If $i=0$ then $H\cong C_{p^2}\times C_p\times C_p$; otherwise $H\cong C_{p^2}\times C_{p^2}$.
\end{proposition}
\begin{proof}
It is obvious that $H$ has order $p^4$ and that it is abelian. Let $w\in G\setminus H$. Then $w\in Hz^c$ for some $c\in\{1,\dots,p-1\}$, and it is easy to check that if $\l\ne 0$ then $w$ centralizes no element of $H$ outside $\langle x,n\rangle$. Since $|\Cent_G(w)\cap H|=p^2$ we must have $|\Cent_G(w)|\le p^3$, and so $w$ lies in no abelian subgroup of order $p^4$; it follows that $H$ is the only such subgroup. If $i\ne 0$ then $H=\langle x,y\rangle$ and so
$H\cong C_{p^2}\times C_{p^2}$. But if $i=0$ then $\langle x^p,y,n\rangle$ is elementary abelian of order $p^3$, and so $H\cong C_{p^2}\times C_p\times C_p$.
\end{proof}

\begin{proposition}\label{prop:powersubgroup}
If $i=j=0$ then the subgroup of $p$-th powers in $G$ is $\langle x^p\rangle$ of order~$p$; otherwise it is $\langle x^p,n\rangle$ of order $p^2$.
\end{proposition}
\begin{proof}
This is obvious from the presentation for $G$.
\end{proof}

\begin{proposition} \label{prop:split}
If $G$ decomposes as a non-trivial direct product then one of the following statements holds.
\begin{enumerate}
\item $(i,j,k,\l)=(0,0,0,0)$. In this case $G\cong Q \times C_p$.
\item $\l=0\ne k$. In this case $G\cong C_{p^2}\times E$, where $E$ is an extraspecial group of order $p^3$. The group $E$ has exponent $p$ if and only $i=j=0$.
\end{enumerate}
\end{proposition}
\begin{proof}
Suppose that $G\cong X\times Y$, where $X$ and $Y$ are non-trivial.

We observe first that a $p$-group of class $2$ whose derived group is elementary abelian of rank $2$ must have order at least $p^5$. Since
the derived group $(X\times Y)'$ is equal to $X'\times Y'$, and since groups of order $p^2$ are abelian, it follows that $|G'|=p$. Since the element $x^pn^k$ is a non-trivial element of $G'$, it must generate $G'$. Now $n^\l$ lies in~$G'$, and so it follows that $\l=0$. This has the consequence, by Proposition \ref{prop:centre}, that $Z(G)$ is isomorphic to $C_{p^2}\times C_p$. Since $Z(X\times Y)=Z(X)\times Z(Y)$, we see that one of these summands---we may suppose it is $X$ without loss of generality---has a centre isomorphic to $C_{p^2}$.  

Suppose first that $|X|=p^4$ and $|Y|=p$. Then $G$ has a central element $u$ of order $p$ which is not a $p$-th power in $G$. But since $u\in\langle x^p,n\rangle$, this implies that $n$ is not a $p$-th power in $G$, and hence that $i=j=0$. We see too that $G'$ is contained in $X$, and so $x^pn^k\in X$. But the central elements of order $p$ in $X$ are those of $\langle x^p\rangle$, and so we must have $k=0$. So all parameters are $0$ in this case, and we see that $G\cong \langle x,y,z\rangle \times \langle n\rangle$, with the first summand being isomorphic to $Q$. 

We cannot have $|X|=p^3$, since $X$ has a centre of order $p^2$. So the only remaining possibility for the orders of $X$ and $Y$ is that $|X|=p^2$ and $|Y|=p^3$. Now since $Y$ is non-abelian, it is extraspecial. Since $G'=\langle x^pn^k\rangle$ is contained in $Y$, it is clear that $k\ne 0$. If $Y$ has exponent $p$ then there is no element $w$ of $G$ such that $w^p=x^pn^k$, and it clearly follows that $i=j=0$. 

We conclude the proof by showing that the stated values for the parameters give decomposable groups $G$. The case $(i,j,k,\l)=(0,0,0,0)$ is obvious. Suppose that $\l=0\ne k$. Then it is easily checked that $G$ decomposes as
\[
G=\langle x \rangle \times \langle x^iy^k, x^jz^k, x^pn^k \rangle.
\]
The second summand is extraspecial, and has exponent $p$ if and only if $i=j=0$.
\end{proof}

\subsection{Parameters}

A candidate extension of $Q$ may have several distinguished presentations with differing parameters. For each isomorphism class of extensions we shall isolate a particular set of parameters. It will be helpful for this purpose to to introduce some terminology. Let $(x,y,z,n)$ and $(x'y'z'n')$ be the generators in two distinguished presentations of a group $G$, with parameters $(i,j,k,l)$ and $(i',j',k',\l')$ respectively. We say that the map $(x,y,z,n)\mapsto(x',y',z',n')$ induces the change of parameters $(i,j,k,l)\mapsto(i',j',k',\l')$.  We shall always treat parameters as elements of~$\Z_p$, since they occur in the presentation (\ref{eq:G16extension}) only as indices to elements of order $p$. 

Some particular generator maps, with the associated change of parameters, are recorded in Table \ref{table:generatormaps}. It is a routine matter in each case to verify that the specified values of $(x',y'z',n')$ give generating sets for $G$, and that they afford a distinguished presentation with the claimed parameters $(i',j',k',\l')$.

\begin{table}
\caption{\label{table:generatormaps}Some generator maps with the associated changes of parameters.}
\begin{tabular}{llll}\\
Map & $(x',y',z',n')$ & $(i',j',k',\l')$ & Conditions \\ \hline
$A(\lambda)$ & $(x,y,z,n^{\lambda^{-1}})$ & $(\lambda i,\lambda j,\lambda k,\lambda\l)$ & $\lambda\ne 0$ \\
$B(\lambda)$ & $(x^\lambda,y^\lambda,z,n)$ & $(\lambda i,j,\lambda k,\lambda\l)$ & $\lambda\ne 0$ \\
$C(\lambda)$ & $(x^\lambda,y,z^\lambda,n)$ & $(i,\lambda j,\lambda k,\lambda^2\l)$ & $\lambda\ne 0$ \\
$D$ & $(x,y,y^{-j/i}z,n)$ & $(i,0,k,\l)$ & $i\ne 0$ \\
$E$ & $(x,z,y^{-1},n)$ & $(j,-i,k,0)$ & $\l=0$
\end{tabular}
\end{table}

\begin{lemma}\label{lemma:parameters}
Let $G$ be a candidate extension of $Q$. Let $\alpha$ be a quadratic non-residue modulo $p$. Then $G$ has a distinguished presentation with parameters in Table \ref{tab:Qparameters}.
\end{lemma}

\begin{table}
\caption{Parameters for candidate extensions of $Q(p)$.}\label{tab:Qparameters}
\begin{tabular}{lll}
Presentation & $(i,j,k,\l)$ \\ \hline
$P_0$	& $(0,0,0,0)$ \\
$P_1$	& $(0,0,1,0)$ \\
$P_2$	& $(1,0,1,0)$ \\
$P_3$	& $(1,0,0,0)$ \\
$P_4$	& $(0,0,0,1)$ \\
$P_5$	& $(0,0,1,1)$ \\
$P_6$	& $(0,1,0,1)$ \\
$P_7$	& $(0,1,1,1)$ \\	
$P_8(\lambda)$	& $(1,0,0,\lambda)$ & $\lambda\in\{1,\alpha\}$ \\
$P_9(\lambda)$ & $(1,0,1,\lambda)$ & $\lambda\in\Z_p^\times$
\end{tabular}
\end{table}

\begin{proof}
Suppose that $G$ has a distinguished presentation $P$ with parameters $(i,j,k,\l)$. If $\l=0$ then it is a simple matter to transform $P$ into one of $P_0$,
$P_1$, $P_2$ or $P_3$ using the maps from Table \ref{table:generatormaps}, for suitable values of $\lambda$. Similarly, if $\l\ne 0$ and $i=0$ then $P$ can be transformed into one of $P_4$, $P_5$, $P_6$ or $P_7$ using these maps. Suppose that $\l$ and~$i$ are both non-zero. Suppose that $k=0$. There exists an element $\lambda\in\Z_p^\times$ such that $\lambda^2\in\{i\l^{-1}, \alpha il^{-1}\}$, and applying the transformations $C(\lambda)$, $A(i^{-1})$ and $D$ transforms~$P$ to $P_8(\lambda)$ for some $\lambda$. Otherwise, if $k\ne 0$, then applying the transformations $C(ik^{-1})$, $A(i^{-1})$ and $D$ transforms $P$ to
$P_9$, for $\lambda=i\ell k^{-1}$.
\end{proof}

\subsection{Non-isomorphic presentations}

We shall show that the presentations listed in Lemma \ref{lemma:parameters} are irredundant, in that no two of them give rise to isomorphic extensions of
$Q$. The observations collected in the following two propositions contain all that is required for this purpose.

\begin{proposition}\label{prop:invariants}
Let $G$ be a candidate extension.
\begin{enumerate}
\item If $G$ has a distinguished presentation with parameter $\l = 0$ then all of its distinguished presentations have $\l=0$.
\item If $G$ has a distinguished presentation with $\l\ne 0$ and with $i=0$, then all of its distinguished presentations have $i=0$.
\item If $G$ has a distinguished presentation with $i=j=0$, then all of its distinguished presentations have $i=j=0$.
\item If $G$ has a distinguished presentation with parameter $k = 0$ then all of its distinguished presentations have $k=0$.
\end{enumerate}
\end{proposition}
\begin{proof}
The first three parts follow immediately from Propositions \ref{prop:centre}, \ref{prop:abeliansubgroup} and \ref{prop:powersubgroup}.
We prove the fourth part here.  

If $\l=0$ then it is clear from Proposition \ref{prop:split} that $G$ has direct factors of order $p^2$ and $p^3$ if and only if $k\ne 0$.
So we may suppose that $\l\ne 0$. By Proposition \ref{prop:abeliansubgroup}, $G$ has a unique maximal abelian subgroup $H$. Suppose that $i=0$; then $H$ is isomorphic to $C_{p^2}\times C_p\times C_p$.
If $A$ is the elementary abelian subgroup of $H$ of order $p^3$; then $A$ is a characteristic subgroup of $G$. It is easy to check that $[G,A]=\langle x^pn^k \rangle$, and this lies in the Frattini subgroup of $H$ if and only if $k=0$. So we suppose instead that $i\ne 0$. Then $H\cong C_{p^2}\times C_{p^2}$. Let $A$ be the elementary abelian subgroup of $H$ of order $p^2$; this is the subgroup of $p$-th powers in $H$. For a fixed element
$w\in G\setminus H$ there is a well defined map $\psi_w:A\longrightarrow A$ given by
\[
\psi_w(h^p) = [h,w] \quad h\in H.
\]
If $A$ is considered as a vector space over $\F_p$, then $\psi_w$ is linear, and it is not hard to calculate that $\tr \phi_w =0$ if and only if $k=0$,   irrespective of the choice of $w$ (which may be assumed to be $z^t$ for some $t\ne 0$). This completes the proof.
\end{proof}

\begin{proposition}\label{prop:P8P9}
The groups given by the presentations $P_8(1)$ and $P_8(\alpha)$ are non-isomorphic. The groups given by the presentations $P_9(\lambda)$ and $P_9(\mu)$ are non-isomorphic whenever $\lambda\ne \mu$.
\end{proposition}
\begin{proof}
Suppose that $G$ has a distinguished presentation with parameters $(1,0,k,\l)$, where $\l\ne 0$. Let $A$ be the subgroup and $\psi_w$ the linear map $A\longrightarrow A$, which are defined in the proof of Proposition \ref{prop:invariants}. It is straightforward to calculate that $\tr \phi_w =\tau k$
and $\det \phi_w = -\tau^2 \l$, for some $\tau\in\F_p^\times$ which depends on the choice of $w$. It is clear from the expression for the determinant that whether or not $\l$ is a quadratic residue modulo $p$ is an invariant of the group. Furthermore, if we suppose that $k=1$ then the value of~$\tau$ is determined, and so $\l$ itself is an invariant. The proposition follows immediately from these observations.
\end{proof}

\begin{lemma}\label{lemma:nonisomorphic}
Let $G$ and $H$ be candidate extensions with distinguished presentations from Table \ref{tab:Qparameters}. If their parameters are not the same, then $G$ and $H$ are not isomorphic.
\end{lemma}
\begin{proof}
This follows from Propositions \ref{prop:invariants} and \ref{prop:P8P9}.
\end{proof}

\subsection{Exceptional extensions}\label{s:excextQ}

\begin{lemma}\label{lemma:exceptionalcondition}
Let $G$ be a candidate extension of $Q$. Then $G$ is exceptional (that is, $\mu(G)<\mu(Q)$) if and only if $G$ has two subgroups $K_1$ and $K_2$, each of order $p^3$, such that $K_1\cap K_2$ has order $p$ and is non-central in $G$.
\end{lemma}
\begin{proof}
Since $Z(G)$ has rank $2$, Theorem \ref{thm:johnson} tells us that a faithful permutation representation of $G$ of minial degree has exactly two orbits. Let $K_1$ and $K_2$ be two non-conjugate point stabilizers. Then $\mu(G)$ = $|G:K_1||G:K_2|$. If $G$ us exceptional, then since $\mu(Q)=p^3$, we see that $K_1$ and $K_2$ must each have order at least $p^3$. So the intersection $Y=K_1\cap K_2$ is non-trivial. Since the representation of $G$ is faithful, it is clear that $Y\cap Z(G)$ is trivial. Since all $p$-th powers in $G$ lie in $Z(G)$ it follows that $Y$ has exponent $p$. Now $YV$ has order $p^2|Y|$, and since it is clear that $G$ has no elementary abelian subgroup of rank $4$, we must have $|Y|=p$. It follows that $|K_1|=|K_2|=p^3$. 

For the converse, suppose that $K_1$ and $K_2$ are subgroups of $G$ of order $p^3$ such that $Y=K_1\cap K_2$ is a non-central subgroup of order $p$. Then $Y$ is not normal in $G$, and the permutation representation of $G$ on the cosets of $K_1$ and $K_2$ is faithful. It follows that $\mu(G)\le 2p^2$, and so $G$ is exceptional.
\end{proof}

\begin{proposition}\label{prop:ExP0P1P2}
The candidate extension with presentation $P_0$ is not exceptional.
The candidate extensions with presentations $P_1$ and $P_2$ are exceptional.
\end{proposition}
\begin{proof}
These are the presentations which give rise to decomposable extensions, by Proposition \ref{prop:split}.
The presentation $P_0$ is afforded by $Q\times C_p$, which is clearly not exceptional. Each of $P_1$ and $P_2$ gives a group
$G\cong C_{p^2}\times E$, where $E$ is extraspecial of order $p^3$. Since both of the extraspecial groups of order $p^3$ have permutation representations of degree $p^2$, we see that $\mu(G)=2p^2$ in each case, and so $G$ is exceptional.
\end{proof}
The groups with presentations $P_1$ and $P_2$ appear in Table \ref{table:results} as $E_1$ and $E_2$ respectively. 

\begin{proposition}\label{prop:ExP3}
The candidate extension with presentation $P_3$ is not exceptional.
\end{proposition}
\begin{proof}
Let $G$ have presentation $P_3$. Then $x$ is central, and $G'=\langle x^p\rangle$. Suppose that~$K$ is a subgroup of order $p^3$, such that
$x^p\notin K$. Then $K$ is abelian, and $G=K\langle x\rangle$. But $G$ is not abelian, which is a contradiction. Thus every subgroup of $G$ of order $p^3$ contains the central element $x^p$, and so by Lemma \ref{lemma:exceptionalcondition} $G$ is not exceptional.
\end{proof}

\begin{proposition}\label{prop:ExP4P5}
The candidate extensions with presentations $P_4$ and $P_6$ are not exceptional.
\end{proposition}
\begin{proof}
These are presentations for which $\ell\neq 0$ and $i=0$. So by Proposition \ref{prop:abeliansubgroup} the candidate subgroup $G$ has an
abelian subgroup $H=\langle x,y,n\rangle$ which has the structure $C_{p^2}\times C_p\times C_p$. Suppose that $K$ is a subgroup of $G$ of order $p^3$ which does not contain~$x^p$. Then it is easy to see that $Y=K\cap H$ is elementary abelian of order $p^2$. Now $Y$ contains an element of $H\setminus V$, whose centralizer in $G$ is $H$. It follows that $K$ is not abelian. But since $k=0$ for these presentations, we see that $[Y,G]=\langle x^p\rangle$, and so $x^p\in K$, a contradiction. Hence the central element $x^p$ lies in every subgroup of $G$ of order $p^3$.
\end{proof}

\begin{proposition}\label{prop:ExP6P7}
The candidate extensions with presentations $P_5$ and $P_7$ are exceptional.
\end{proposition}
\begin{proof}
It is easy to check that the subgroups $K_1=\langle x,y\rangle$ and $K_2=\langle x^pn^k,y,z\rangle$ satisfy the conditions stated in Lemma \ref{lemma:exceptionalcondition}.
\end{proof}
The groups with presentations $P_5$ and $P_7$ appear in Table \ref{table:results} as $E_3$ and $E_4$ respectively. 

\begin{proposition}\label{prop:ExP8P9}
The candidate extension with presentation $P_8(\lambda)$ is exceptional if $\lambda=1$, but not if $\lambda=\alpha$. The candidate extension with presentation $P_9(\lambda)$ is exceptional if and only if $1+4\lambda$ is a non-zero quadratic residue modulo $p$.
\end{proposition}
\begin{proof}
These are the cases for which the parameters $i$ and $\ell$ are non-zero. Let $H$ be the subgroup form Proposition \ref{prop:abeliansubgroup}, and let $w\in G\setminus H$. We show that $G$ is exceptional if and only if the map $\psi_w$ defined in the proof of Proposition \ref{prop:invariants} is diagonalizable. (This does not depend on the choice of $w$, which affects $\psi_w$ only up to scalar multiplication.) 

Let $G$ be a candidate extension with one of the stated presentations. Let~$K_1$ and~$K_2$ be subgroups of $G$ of order $p^3$ whose intersection has order $p$ and is non-central. Since~$H$ has the structure $C_{p^2}\times C_{p^2}$, we see that $K_1$ and~$K_2$ cannot have exponent $p$. It follows that $K_1\cap H$ and $K_2\cap H$ are cyclic of order~$p$. Let $u_1$ generate $K_1\cap H$ and let $u_2$ generate $K_2\cap H$. Then $u_1^p$ and $u_2^p$ are linearly independent elements of~$V$, considered as a vector space. Choose $w$ to be a generator of $K_1\cap K_2$. Note that $w\notin H$, since the elements of order $p$ in $H$ are central in $G$. Now we observe that both $u_1^p$ and $u_2^p$ must be eigenvectors of the map $\psi_w$. 

Conversely, suppose that $w\in G\setminus H$, and suppose that $u_1$ and $u_2$ are elements of~$H$ such that $u_1^p$ and $u_2^p$ are linearly independent eigenvectors of $\psi_w$. Then it is easy to check that the subgroups $K_1=\langle u_1,w\rangle$ and $K_2=\langle u_2,w\rangle$ satisfy the conditions of Lemma \ref{lemma:exceptionalcondition}, and so $G$ is exceptional. 

Since the map $\psi_w$ is not scalar for any $w\in G\setminus H$, it follows that $G$ is exceptional if and only if the characteristic polynomial of $\psi_w$ has distinct roots in $\F_p$. Taking $w=z$, we see that the characteristic polynomial is $X^2-kX-\ell$, which has distinct roots if and only if $k^2+4\ell$ is a non-zero square in $\F_p$. Now substituting the appropriate values for $k$ and $\ell$ establishes the proposition in all cases.
\end{proof}
The group with presentations $P_8(1)$ appears in Table \ref{table:results} as $E_5$. For some values of~$\lambda$ the group with presentation $P_9(\lambda)$ is exceptional, and appears there as $E_6(\lambda)$, with the appropriate conditions on $\lambda$.

\subsection{Counting exceptional extensions}

\begin{proposition} \label{numberofexceptionalextensionsQ}
The number of groups $G$ which are exceptional extensions of $Q$, up to isomorphism, is $(p+7)/2$.
\end{proposition}
\begin{proof}
There is an exceptional extension given by the presentation $P_9(\lambda)$ whenever $\lambda\neq 0$ and $1+4\lambda$ is a quadratic residue for $p$.
Every quadratic residue for $p$ except $1$ is expressible as $1+4\lambda$ for some non-zero $\lambda$, and so there are $(p-3)/2$ such presentations.
The five groups with presentations $P_1$, $P_2$, $P_6$, $P_7$ and $P_8(1)$ are also exceptional, giving the total stated in the proposition.   
\end{proof}

\section{Extensions of $Q_1(p)$ and $Q_\alpha(p)$}\label{s:Q_1}

Let $p$ be a prime greater than $3$, and let $Q_\zeta$ be the group of order $p^4$ given by
\[
Q_\zeta=\langle x,y,z \mid x^{p^2}=y^p= [x,y]=1, z^p=x^{\zeta p},  [x,z] = y, [y,z]=x^{\zeta p} \rangle,
\]
where $\zeta$ is an integer coprime with $p$. The isomorphism class of $Q_\zeta$ depends only on whether or not $\zeta$ is a
quadratic residue modulo $p$. The particular values $\zeta=1$ and $\zeta=\alpha$ give the presentations defining the groups 
$Q_1(p)$ and $Q_\alpha(p)$ respectively. 

A central extension of $Q_\zeta$ by a group of order $p$ has the form
\begin{align}\label{eq:G25extension}
G &=\langle x,y,z,n \mid n^p=1, x^{p^2}=n^h, y^p=n^i, z^p=x^{\zeta p}n^j,\ \textrm{$n$ central}, \nonumber \\
  &\quad \quad  [x,z]=yn^\l, [x,y] = n^{-m}, [y,z]=x^{\zeta p}n^{\zeta m+k}\rangle
\end{align}
for integers $h,i,j,k,\l,m\in\{0,\dots, p-1\}$. The generators of $G$ in this presentation have been so labelled that the images of $x,y,z$, in the quotient of $G$ by its central subgroup $\langle n\rangle$, correspond to the generators $x,y,z$ of $Q_\zeta$; the exponents of $n$ have been chosen for convenience at a later point of the argument. 

We note that $G$ has nilpotency class $3$, and that its derived subgroup has exponent at most $p^2$. So Proposition \ref{prop:ppower} applies, and we have that the $p$-power map on $G$ is an endomorphism. 

It is clear that a subgroup $H$ of $Q_\zeta$ not containing the socle $\langle x^p\rangle$ has order at most~$p$, from which it is clear that $\mu(Q_\zeta)=p^3$. Suppose that $\mu(G)<\mu(Q_\zeta)$; then it is clear that $G$ can have no element of order $p^3$. In particular, the generator $x$ has order $p^2$, and so we must have $h=0$. 

It is easy to see that the centre $Z$ of $G$ is $\langle x^p,n\rangle$, and that $G/Z$ is an extraspecial group with centre $\langle yZ\rangle$.
Any subgroup of $G/Z$ of order greater than $p$ contains $yZ$, and it follows easily that any subgroup of $Z$ of order greater than $p^2$ contains
$y^p$. So~$n^i$ is contained in the kernel of any permutation representation of $G$ of degree less than $p^3$. Hence if $G$ is exceptional, then
$i=0$. 

We may also suppose that $\l=0$, simply by replacing the generator $y$ with $yn^\l$. Thus we may drop three parameters from our notation, and write
$G(j,k,m)$ for the group with the presentation (\ref{eq:G25extension}) with $h=i=\l=0$. These are our \emph{distinguished presentations} in this case, with parameters $(j,k,m)$. We call groups affording such a presentation candidate extensions of $Q_\zeta$; these include all of the
exceptional extensions of $Q_\zeta$.

\begin{proposition} A candidate extension $G$ of $Q_\zeta$ has order $p^5$.
\end{proposition}
\begin{proof} Consider the group 
\[
W=\langle x,y,n \mid x^p=y^p=n^p=[x,n]=[y,n]=1, [x,y]=n^{-m}\rangle.
\]
It is clear that $W$ has order $p^3$ and exponent $p$.
We observe that $[xy,yn^{\zeta m+k}]=[x,y]$ in~$W$, and so there exists an automorphism $z$ of $W$ such that $x^z=xy$ and $y^z=yn^{\zeta m+k}$. It is easy to check that $x^{z^a} = xy^an^{a \choose 2}$ for any integer $a$, and it follows easily that $z$ has order $p$. Now the extension of $W$ by the automorphism $z$
has the presentation
\[
\langle x,y,z,n | x^p=y^p=z^p=n^p=1,\ \textrm{$n$ central}, [x,y]=n^{-m},  [x,z]=y, [y,z]=n^{\zeta m+k}\rangle,
\]
and we see that this is the quotient of $G$ by the subgroup $\langle x^p\rangle$. Since $x^p$ is non-trivial in $Q_\zeta$, it is non-trivial in $G$, and so
$|G|=p|W|=p^5$ as required. \end{proof}

\subsection{Parameters}

If $(x,y,z,n)$ and $(x'y'z'n')$ are the generators in two distinguished presentations of a group $G$, with parameters $(j,k,m)$ and $(j',k',m')$ respectively, then we say that the map $(x,y,z,n)\mapsto(x',y',z'n,')$ induces the change of parameters $(j,k,m)\mapsto(j',k',m')$. 

Some particular generator maps, with the associated change of parameters, are recorded in Table \ref{table:generatormaps25}. It is a routine matter to verify that the map $A(\lambda)$ gives a distinguished presentation with the parameters claimed. For $B(\lambda)$, we observe that
\[
[x^\lambda,x^{\zeta(\lambda-1)}z] = [x^\lambda,z] = y^\lambda n^{m{\lambda \choose 2}},
\]
with the second equality given by an easy induction. So we have $[x',z']=y'$, and it is straightforward to check that the other relations of a distinguished presentation are satisfied with the parameters stated.

\begin{table}
\caption{\label{table:generatormaps25}Some generator maps with the associated changes of parameters.}
\begin{tabular}{llll}\\
Map & $(x',y',z',n')$ & $(j',k',m')$ & Conditions \\ \hline
$A(\lambda)$ & $(x,y,z,n^{\lambda^{-1}})$ & $(\lambda j,\lambda k,\lambda m)$ & $\lambda\ne 0$ \\
$B(\lambda)$ & $(x^\lambda,y^\lambda n^{m{\lambda \choose 2}},x^{\zeta(\lambda-1)}z,n)$ & $(j,\lambda k,\lambda^2 m)$ & $\lambda\ne 0$
\end{tabular}
\end{table}

\begin{proposition}\label{proposition:invariants25}
Let $G$ be a candidate extension. If $G$ has a distinguished presentation with one of its parameters $j,k,m$ equal to $0$, then all of its distinguished presentations have that parameter equal to $0$.
\end{proposition}
\begin{proof}
Let $G$ have a distinguished presentation with parameters $j,k,m$. We observe that if $j=0$ then the image $J$ of the endomorphism $g\mapsto g^p$
is the subgroup $\langle x^p\rangle$ of order $p$, whereas if $j\neq 0$ then this image coincides with the centre $Z$, of order $p^2$. If $m=0$ then
the derived group $G'$ is $\langle y,x^{\zeta p}n^k\rangle$ of order $p^2$, whereas if $m\ne 0$ then $G'=\langle y,n,x^p\rangle$ of order $p^3$. Clearly these are invariants of the group $G$. 

For the parameter $k$, let $w$ be any non-central element of $G$ contained in the derived group $G'$. (The generator $y$ in any distinguished presentation for $G$ is such an element.) We consider separately the cases that $j=0$ and  that $j\ne 0$. When $j=0$, the image $J$ of the $p$-power endomorphism
is a central subgroup of order $p$. The map $G\longrightarrow Z/J$ given by $g\mapsto[w,g]J$ is a homomorphism with
kernel $\langle x^{\zeta m+k}z^{-m},y,x^p,n\rangle$. This kernel is independent of the element $w$, and we observe that it has exponent $p$ if and only if
the first of the given generators has order $p$; this occurs if and only if $k=0$.

Suppose next that $j\ne 0$. By applying the generator map $A(1/j)$ from Table~\ref{table:generatormaps25}, we see that we may assume that $j=1$.
We have $J=Z$, and so every central element has a $p$-th root in $G$. We observe that the map
$\psi_w:Z\longrightarrow Z$ given by $\psi_w(u^p)=[w,u]$ is a well defined homomorphism, since $w$ is centralized by the kernel $\langle y,x^p,n\rangle$ of the $p$-power endomorphism. We may consider $Z$ as a $2$-dimensional vector space over~$\F_p$, and now the map $\psi_w$ becomes a linear transformation
of $Z$. Let $t$ be such that $w\in y^tZ$. Then with respect to the basis $(x^p,n)$ for $Z$, the map $\psi_w$ has the matrix
\[
M_t = t\left(\begin{array}{cc} 0&\zeta \\ m&k \end{array}\right).
\]
Now this map has trace $0$ if and only if $k=0$, regardless of the choice of $w$.
\end{proof}

\begin{proposition}\label{proposition:legendre}
Let $G$ have two distinguished presentations, each with parameter $j=1$. Let $m_1$ and $m_2$ be the values for the parameter $m$ in these presentations. Then $\left(\frac{m_1}{p}\right) = \left(\frac{m_2}{p}\right)$, where $\left(\frac{a}{p}\right)$ is the Legendre symbol.
\end{proposition}
\begin{proof}
By Proposition \ref{proposition:invariants25} we may suppose that $m_1$ and $m_2$ are both non-zero. Since $j=1$, we may construct the transformation
$\psi_w$ from the proof of Proposition \ref{proposition:invariants25}. We see that $\det \psi_w = -t^2\zeta m$, which is independent of $w$ modulo squares in
$\Z_p^{\times}$.
\end{proof}

\begin{proposition}\label{proposition:kparameter}
Let $G$ have two distinguished presentations, each with parameter $j=1$, and with the same non-zero value for the parameter $m$.
Let $k_1$ and $k_2$ be the values for the parameter $k$ in these presentations. Then $k_2=\pm k_1$.
\end{proposition}
\begin{proof}
Since $j=1$, we can construct the transformation $\psi_w$ from the proof of Proposition \ref{proposition:invariants25}. Now we see that
\[
\frac{(\tr \psi_w)^2}{\det \psi_w} = -\frac{k^2}{\zeta m},
\]
and so we must have $k_1^2=k_2^2$.
\end{proof}

\begin{lemma}\label{lemma:parameters25}
Let $G$ be a candidate extension of $Q_\zeta$. Let $\alpha$ be a quadratic non-residue modulo $p$. Then $G$ has a distinguished presentation with parameters in Table \ref{tab:Qzetaparameters}.
\end{lemma}

\begin{table}
\caption{Parameters for candidate extensions of $Q_\zeta$.}\label{tab:Qzetaparameters}
\begin{tabular}{lll}
Presentation & $(j,k,m)$ \\ \hline
$P_0$	& $(0,0,0)$ \\
$P_1$	& $(0,0,1)$ \\
$P_2$	& $(0,1,0)$ \\
$P_3$	& $(1,0,0)$ \\
$P_4$	& $(0,1,1)$ \\
$P_5$	& $(1,1,0)$ \\
$P_6(\lambda)$& $(1,\lambda,1)$ & $\lambda\in\{0,\dots,(p-1)/2\}$)\\
$P_7(\lambda)$& $(1,\lambda,\alpha)$ & $\lambda\in\{0,\dots,(p-1)/2\}$)
\end{tabular}
\end{table}

\begin{proof}
It follows from Propositions \ref{proposition:invariants25}, \ref{proposition:legendre} and \ref{proposition:kparameter} that no two of the sets of
parameters listed give rise to isomorphic groups. It is a routine matter to show that any set of parameters may be transformed into one of the listed sets by means of the generator maps listed in Table~\ref{table:generatormaps25}.
\end{proof}

\subsection{Exceptional extensions}

\begin{proposition}\label{Exceptionalwhenj=0}
Let $G$ be a candidate extension with parameter $j=0$. Then $G$ is exceptional if and only if parameter $k\ne 0$.
\end{proposition}
\begin{proof}
The group $G$ is exceptional if and only if has subgroups $H_1$ and $H_2$, each of order $p^3$, such that the intersection
$H_1\cap H_2\cap Z$ is trivial. 

The image $J$ of the $p$-power endomorphism is generated by $x^p$. It follows that any subgroup of $G$ not containing $x^p$ has exponent $p$. Now the
kernel of the $p$-power endomorphism is $K=\langle y,x^{-\zeta}z,x^p,n\rangle$, which factorizes as
\[
\langle y,x^{-\zeta}z,x^{\zeta p}n^k\rangle\times \langle n\rangle.
\]
Let $L$ be the factor $\langle y,x^{-\zeta}z,x^{\zeta p}n^k\rangle$. Then $L$ is extraspecial, with centre $\langle x^{\zeta p}n^k\rangle$.
Let~$H$ be a subgroup of $K$ of order $p^3$. Then $|H\cap L|\ge p^2$, and so $x^{\zeta p}n^k\in H$. It follows that if $k=0$, then every subgroup of $G$ of order $p^3$ contains the central element $x^p$, and hence $G$ is not exceptional.

Conversely, if $k\ne 0$, then it is easily checked that the subgroups
\[
H_1= \langle y,x^{-\zeta}z, x^{\zeta p}n^k\rangle,\quad H_2=\langle x^{-(\zeta m+k)}z^m, y \rangle,
\]
both have order $p^3$, and their intersection is the non-central subgroup $\langle y\rangle$. Hence $G$ is exceptional in this case.
\end{proof}
Proposition \ref{Exceptionalwhenj=0} tells is that the group with presentations $P_2$ and $P_4$ are exceptional. Taking $\zeta$ to be either $1$ or $\alpha$, these groups appear in Table \ref{table:results} as $F^{(\zeta)}_1$ and $F^{(\zeta)}_2$ respectively.

\begin{proposition}\label{Exceptionalwhenjnot0}
Let $G$ be a candidate extension with parameter $j=1$, and parameters $k, m$. Then $G$ is exceptional if and only if $k^2+4\zeta m$ is a non-zero quadratic residue modulo $p$.
\end{proposition}
\begin{proof}
Let $H_1$ and $H_2$ be subgroups of $G$ of order $p^3$, such that $H_1\cap H_2\cap Z$ is trivial. It is clear that $H_1\cap H_2$ has exponent $p$ (since the image $J$ of the $p$-power endomorphism is central). Since $j\ne 0$, the kernel of the $p$-power endomorphism is $\langle Z,y\rangle$, and so
$H_1\cap H_2$ contains an element $w\in yZ$. We may now construct the map $\psi_w$ from the proof of Proposition \ref{proposition:invariants25}. (The value of $t$ is quite unimportant, but as it happens our choice of $w$ ensures that $t=1$.) 

It is clear, since any subgroup of $G$ of order $p^3$ intersects the centre non-trivially,
that each of $H_1\cap Z$ and $H_2\cap Z$ must have order $p$. Let these intersections be generated by $u_1$ and $u_2$ respectively.
In particular, neither $H_1$ nor $H_2$ can be contained in the kernel of the $p$-power endomorphism, and so each has exponent $p^2$. It follows that
$u_1$ and $u_2$ have $p$-th roots $v_1\in H_1$ and $v_2\in H_2$ respectively. It follows from the definition of the map $\psi_w$
that $\psi_w(u_1)=[w,v_1]\in H_1\cap Z$, and  $\psi_w(u_2)=[w,v_2]\in H_2\cap Z$. Hence each of $u_1$ and $u_2$ is an eigenvector of the map $\psi_w$.

As a converse, we note that if $\psi_w$ has linearly independent eigenvectors $u_1$ and $u_2$, with $p$-th roots  $v_1$ and $v_2$ in $G$ respectively,
then the subgroups $H_1=\langle v_1,w\rangle$ and $H_2=\langle v_2,w\rangle$ both have order $p^3$, and their intersection $\langle w\rangle$ is non-central. Hence $G$ is exceptional if and only if $\psi_w$ is diagonalizable. The proposition now follows from the observations that the
discriminant of $\psi_w$ is $k^2+4\zeta m$, and that since $\psi_w$ is not scalar, a discriminant of $0$ implies that $\psi_w$ is not diagonalizable.
\end{proof}
Let $\zeta\in\{1,\alpha\}$. The presentation $P_5$ gives an exceptional group, which appears in Table \ref{table:results} as $F^{(\zeta)}_3$.
For suitable values of $\lambda$ the groups with presentations $P_6(\lambda)$ or $P_7(\lambda)$ are exceptional; these appear
as $F^{(\zeta)}_4(\lambda)$ and $F^{(\zeta)}_5(\lambda)$ respectively, with the appropriate conditions on $\lambda$.

\subsection{Counting exceptional extensions}

\begin{proposition} \label{numberofexceptionalextensionsQ_1}
The number of groups $G$ which are exceptional extensions of $Q_\zeta$, up to isomorphism, is $(p+5)/2$.
\end{proposition}
\begin{proof}
From Propositions \ref{Exceptionalwhenj=0} and \ref{Exceptionalwhenjnot0} we see that the presentations $P_2$, $P_4$ and $P_5$ give exceptional extensions, whereas $P_0$, $P_1$ and $P_3$ do not. It remains to deal with the presentations $P_6(\lambda)$ and $P_7(\lambda)$. 

Working in the set $\Z_p$ of integers modulo $p$, let $t$ be non-zero, and let \[
S_t:=\{(r,\lambda) \mid r\neq 0, t=r^2-\lambda^2\}.
\]
Then it is clear that $|S_t|$ is the number of (ordered) factorizations of $t$ as a product of distinct elements of $\Z_p$.
This number is $p-3$ if $t$ is a square, and $p-1$ otherwise. Now define
\[
T=\{(\mu,r,\lambda) \mid \mu\in\{1,\alpha\}, r\neq 0, r^2-\lambda^2=4\zeta \mu \}.
\]
Since $4\zeta$ and $4\zeta\alpha$ are non-zero, and  exactly one of them is a square, we see that $|T|=2p-4$. Each element of $T$ yields a distinguished presentation with parameters $(j,k,m) = (1,\pm\lambda,\mu)$ (where the sign in the second parameter is chosen to put it in the range $\{0,\dots,\frac{p-1}{2}\}$). Since the discriminant $k^2+4\zeta m$ for this presentation is $r^2$, Proposition
\ref{Exceptionalwhenjnot0} tells us that the corresponding extension is exceptional; moreover it is clear that all the
exceptional extensions with presentations $P_6(\lambda)$ or $P_7(\lambda)$ arise in this way. Now two elements $(\mu,r,\lambda)$ and $(\mu',r',\lambda')$ of $T$ yield the same presentation if and only if $\mu=\mu'$, $r=\pm r'$ and $\lambda=\pm\lambda'$. Exactly two elements of $T$ have
$\lambda=0$. So the number of exceptional extensions arising from $T$ is $1+(|T|-2)/4$, which is $(p-1)/2$. Together with the three exceptional extensions
already found, this makes up the number $(p+5)/2$ stated in the proposition.
\end{proof}

\subsection{Isomorphisms of exceptional extensions of $Q_1(p)$ and $Q_\alpha(p)$}\label{s:isomorphisms}

Recall that~$\alpha$ is a quadratic non-residue modulo $p$. In this section we establish the following fact. 
\begin{proposition}\label{prop:isomorphisms}
Let $G$ be a group of order $p^5$. Then $G$ is an exceptional extension of $Q_\alpha$ if and only if it is an exceptional extension of $Q_1$.
\end{proposition}

We  show directly that every exceptional extension of $Q_\alpha$ is isomorphic to
an exceptional extension of $Q_1$. Since $Q_\alpha$ and $Q_1$ have the same number of exceptional extensions by
Proposition~\ref{numberofexceptionalextensionsQ_1}, it follows that the converse also holds.

\begin{proposition}\label{Isomorphismswhenj=1}
Let $G$ be an exceptional extension of $Q_\alpha$ affording a distinguished presentation with parameters $(j,k,m)$, where $j=1$. Then $G$ is an exceptional extension of $Q_1$.
\end{proposition}
\begin{proof}
Pick $w\in G$ such that $w\notin Z(G)$ and $w^p=1$. Then the map $\psi_w$ is diagonalizable by Proposition~\ref{Exceptionalwhenjnot0}. So there exist
elements $u$ and $v$ of $Z(G)$ such that $\psi_w(u)=u^\gamma$ and $\psi_w(v)=v^\delta$, for some integers $\gamma,\delta$. Let $s^p=u$ and $t^p=v$.
Then $[s,t]\in w^iZ(G)$ for some $i$ coprime with $p$. If $ij\cong 1 \bmod p$ then it is easy to check that $[s^j,t]\in wZ(G)$. Replacing $s$ with $s^j$,
and then replacing $w$ with $[s,t]$ (which does not effect the map $\psi_w$), we see that $G$ has a presentation
\[
G = \langle s,t,w | s^{p^2}=t^{p^2} = w^p =1, [s,t]=w, [w,s]=s^{p\gamma}, [w,t]=t^{p\delta}\rangle.
\]
This presentation is dependent only on the eigenvalues of the map $\psi_w$. But from the matrix for $\psi_w$ given in the proof of
Proposition~\ref{proposition:invariants25}, we see that these depend on $\alpha$ only insofar as they depend on $\alpha m$. So if we take a distinguished
presentation with parameters $(1,k,\alpha m)$ for an extension $E$ of $Q_1$, then it is clear that $E\cong G$. It remains only to note that $E$ is an
exceptional extension of $Q_1$ by Proposition~\ref{Exceptionalwhenjnot0}.
\end{proof}

There are two exceptional extensions of $Q_\alpha$ not covered by Proposition~\ref{Isomorphismswhenj=1}, given by presentations $P_2$ and $P_4$ in
Table~\ref{tab:Qzetaparameters}.

\begin{proposition}
The extension of $Q_\alpha$ with presentation $P_2$ is an exceptional extension of $Q_1$.
\end{proposition}
\begin{proof}
Let $x,y,z,n$ be the generators of a group $G$ given by the presentation $P_2$.
Define elements $\tilde{z}=x^{1-\alpha}z$ and $\tilde{n}=x^{(\alpha-1)p}n^k$. It is easy to check that in terms of the
generators $x,y,\tilde{z},\tilde{n}$, the group $G$ affords the presentation
\begin{align*}
G&= \langle x,y,\tilde{z},\tilde{n} \mid \tilde{n}^p= x^{p^2} =y^p=1, \tilde{z}^p=x^p, [x,\tilde{z}]=y, [x,y] = 1,\\
& \quad \quad [y,\tilde{z}]=x^pn, [x,\tilde{n}]=[y,\tilde{n}]=[\tilde{z},\tilde{n}]=1\rangle.
\end{align*}
This is a distinguished presentation with parameters $(0,1,0)$ for an extension of $Q_1$, and as such is exceptional by
Proposition~\ref{Exceptionalwhenj=0}.
\end{proof}

\begin{proposition}
The extension of $Q_\alpha$ with presentation $P_4$ is an exceptional extension of $Q_1$
\end{proposition}
\begin{proof}
There exist elements $t$ and $b$ of $C_p$ such that $t^2-b^2=4\alpha$. Define
\[
a=\frac{t-(2\alpha+1)b}{2},\quad c=-(\alpha+1)b,\quad d=\frac{t+(2\alpha+1)b}{2\alpha}.
\]
It is straightforward to check that these values satisfy the following equations:
\begin{equation}
a+\alpha b = c+\alpha d = \frac{t-b}{2} \label{Eq1}
\end{equation}
\begin{equation}
ad - bc = 1 \label{Eq2}
\end{equation}

Let $G$ be given by a distinguished presentation $P_4$. Define elements
\[
\tilde{x}=x^az^b,\quad \tilde{z}=x^cz^d,\quad \tilde{y}=[x,z],\quad \tilde{n}=x^{pb}n^{d-b}.
\]
Note that $\tilde{z}^p=\tilde{x}^p = x^{p(t-b)/2}$ by (\ref{Eq1}). Since $t\neq b$, it is clear that $\tilde{x}$ and $\tilde{z}$ have order~$p^2$.
Furthermore, it follows from (\ref{Eq2}) that $\tilde{y}\in yZ(G)$, and so $[\tilde{y},g]=[y,g]$ for all $g\in G$.
We also note that since $4\alpha$ is a quadratic non-residue modulo $p$, we must have $b\neq 0$, and so $\tilde{n}$ is a central element of order $p$.
It is straightforward to check the identities
\[
[\tilde{x},\tilde{y}] = \tilde{n}^{-\alpha}, \quad [\tilde{y},\tilde{z}] = \tilde{x}^p\tilde{n}^{\alpha+1}.
\]
With respect to the generators $\tilde{x},\tilde{y},\tilde{z},\tilde{n}$, the group $G$ affords the presentation
\begin{align*}
G&=\langle \tilde{x},\tilde{y},\tilde{z},\tilde{n} \mid \tilde{n}^p= \tilde{x}^{p^2} =\tilde{y}^p=1, \tilde{z}^p=\tilde{x}^p, [\tilde{x},\tilde{z}]=\tilde{y}, [\tilde{x},\tilde{y}] = \tilde{n}^{-\alpha}, \\
  &\quad \quad [\tilde{y},\tilde{z}]=\tilde{x}^p\tilde{n}^{\alpha+1}, [\tilde{x},\tilde{n}]=[\tilde{y},\tilde{n}]=[\tilde{z},\tilde{n}]=1\rangle.
\end{align*}
This is a distinguished presentation with parameters $(0,1,\alpha)$ for an extension of $Q_1$, and as such is exceptional by
Proposition~\ref{Exceptionalwhenj=0}.
\end{proof}

In the nomenclature of Table \ref{table:results}, we have that $F^{(1)}_1\cong F^{(\alpha)}_1$, $F^{(1)}_2\cong F^{(\alpha)}_2$, and
$F^{(1)}_3\cong F^{(\alpha)}_3$. Each group in $\{F^{(1)}_4(\lambda)\}$ is isomorphic with a group in 
$\{F^{(\alpha)}_5(\lambda)\}$, while each group in $\{F^{(1)}_5(\lambda)\}$ is isomorphic with a group in $\{F^{(\alpha)}_4(\lambda)\}$.

\subsection{The number of exceptional groups of order $p^5$}

It remains only to marshal the various results which together establish Theorem \ref{T:counting}. For $p=2$, the result is from \cite{EasdownPraeger}. For $p=3$, it is given by the computation described in Section \ref{s:computation}. For larger primes it follows easily from Propositions \ref{numberofexceptionalextensionsQ} and \ref{numberofexceptionalextensionsQ_1} together with Proposition \ref{prop:isomorphisms}.

\end{document}